\documentclass[12 pt]{article}
\usepackage{amsmath}
\usepackage{amsthm}
\usepackage{graphicx}
\usepackage{amssymb}

\def\ep{\epsilon}

\def\sub{\subset}

\DeclareMathOperator{\sgn}{sgn}

\newcommand{\R}{{\mathbb{R}}}

\newcommand{\Z}{{\mathbb{Z}}}

\newtheorem{df}{Definition}

\newtheorem{lem}{Lemma}

\newtheorem{rem}{Remark}

\newtheorem{thm}{Theorem}

\title{Exactly self-similar blow-up of the generalized De Gregorio equation}
\author{Fan Zheng
}

\begin{document}

\maketitle
\renewcommand{\thefootnote}{}
\renewcommand{\thefootnote}{\arabic{footnote}}

\begin{abstract}
We study exactly self-similar blow-up profiles fot the generalized De Gregorio model for the three-dimensional Euler equation:
\[
w_t + auw_x = u_xw, \quad u_x = Hw
\]
We show that for any $\alpha \in (0, 1)$ such that $|a\alpha|$ is sufficiently small, there is an exactly self-similar $C^\alpha$ solution that blows up in finite time. This simultaneously improves on the result in \cite{ElJe} by removing the restriction $1/\alpha \in \Z$ and \cite{ElGhMa,ChHoHu}, which only deals with asymptotically self-similar blow-ups.
\end{abstract}

\pagebreak
\tableofcontents
\pagebreak

\section{introduction}\label{Intro}

The famous millennium problem of the global regularity of the motion of incompressible fluids in three-dimensional space concerns the Navier--Stokes equation:
\begin{align*}
u_t + u \cdot \nabla u &= -p + \nu\Delta u,\\
\nabla \cdot u = 0,
\end{align*}
where $u(x, t)$ is the velocity field and $p(x, t)$ is the scalar pressure.
The coefficient $\nu \ge 0$ reflects the viscosity of the fluid.
For inviscid fluid ($\nu = 0$) the equation reduces to the Euler equation.
The vanishing of $\nabla \cdot u$ captures the incompressibility condition.
The global wellposedness of the Euler and Navier--Stokes equations in three dimensions for smooth and decaying initial data is wide open,
attracting a great deal of research efforts.
The interested reader is referred to the surveys \cite{BeMa,Co,Fe,Gi,Ho}.

Let $\omega = \nabla \times u$ be the vorticity, which then satisfy
\[
\omega_t + (u \cdot \nabla)\omega = \omega \cdot \nabla u + \nu\Delta\omega.
\]
The second term on the left-hand side, known as the advection term,
has the effect of transporting the vorticity. Since $u$ is divergence free,
the advection term will not affect the $L^p$ norms of the vorticity.
The first term on the right-hand side, known as the vortex stretching term,
is only present in the three-dimensional case. Schematically $\nabla u \approx \omega$, so the vortex stretching term may be as bad as $\omega^2$ in the worst-case scenario, and may cause blow-up of the equation.
This is thought to be the crux of the millennium problem.

\subsection{The generalized De Gregorio Model}\label{DeG}

In \cite{DeG1,DeG2} De Gregorio proposed a one-dimensional equation to model the competition between advection and vortex stretching in the Euler equation. The equation belongs to the family
\[
w_t + auw_x = u_xw, \quad u_x = Hw
\]
where $H$ denotes the Hilbert transform, and $a$ is a real parameter
quantifying the relative strength of advection, modeled by $uw_x$,
and vortex stretching, modeled by $u_xw$. Note that the relation $u_x = Hw$ also mimics the Biot--Savart law relating the velocity and the vorticity.
It turns out that this equation also models a variety of other equations,
including the surface quasi-geostrophic equation, see \cite{ElJe}.
Other similar 1D models of the Euler equation can be found in \cite{Ki}.

What De Gregorio studied in \cite{DeG1} is the case $a = 1$,
which mirrors the Euler equation. The special case $a = 0$ had appeared in Constantin--Lax--Majda \cite{CLM} and had been known to develop finite-time singularity for all non-trivial initial data. The generalization to arbitrary $a$ was done by Okamoto--Sakajo--Wunsch \cite{OkSaWu}.
When $a < 0$, advection and vortex stretching cooperate to cause a blow-up,
as shown in \cite{CaCo}. When $a > 0$, they fight against each other
and the picture is more interesting. For small $a$, smooth blow-up solutions were found by Elgindi--Jeong \cite{ElJe}, Elgindi--Ghoul--Masmoudi \cite{ElGhMa} and Chen--Hou--Huang \cite{ChHoHu}.
In general, $C^\alpha$ blow-ups were found for $|a\alpha|$ small enough.
A numerical investigation of the behavior of the solution with different values of $a$ can be found in Lushnikov--Silantyev--Siegel \cite{LuSiSi}.

\subsection{The result}\label{Result}

This note improves on the known results on the $C^\alpha$ blow-ups. The ones constructed in \cite{ElJe} is exactly self-similar, i.e., of the form
\[
w(x, t) = \frac{1}{1 - t}W\left( \frac{x}{(1 - t)^{(1+\lambda)/\alpha}} \right)
\]
where $W \in C^\alpha$, but with the restriction that $1/\alpha \in \Z$.
The construction in \cite{ElGhMa,ChHoHu} works for all $\alpha \in (0, 1)$,
as long as $|a\alpha| \ll 1$, but the solution is not exactly self-similar,
but only asymptotically so. In this note we fill the gap by constructing exactly self-similar blow-up solutions for all $\alpha \in (0, 1)$,
provided that $|a\alpha| \ll 1$. Specifically we show that

\begin{thm}\label{Thm}
There is $c > 0$ such that if $|a\alpha| < c$, then there are $W(\cdot;a) \in C^\alpha$ and $\lambda(a) \in \R$ such that
\[
w(x, t) = \frac{1}{1 - t}W\left( \frac{x}{(1 - t)^{1+\lambda(a)}};a \right)
\]
is a self-similar solution. Moreover, $W(\cdot;a)$ and $\lambda(a)$ are analytic in $a$.
\end{thm}

\subsection{The method}\label{Method}

We mostly follow \cite{ElJe}. Plugging the ansatz in the equation
\[
w_t + auw_x = u_xw, \quad u_x = Hw = -|\nabla|^{-1}w_x
\]
we get the steady-state equation for $W$:
\[
F := \left( \frac{1 + \lambda}{\alpha}x + aU \right)W_x + (1 - U_x)W = 0,
\quad U_x = HW.
\]
The explicit solution when $a = 0$ (see (4.2)--(4.3) of \cite{ElJe})
\begin{align*}
(\bar W, \bar\lambda) &= \left( -\frac{2\sin(\alpha\pi/2)|x|^\alpha\sgn x}{|x|^{2\alpha} + 2\cos(\alpha\pi/2)|x|^\alpha + 1}, 0 \right),\\
\bar U_x &= H\bar W = \frac{2(\cos(\alpha\pi/2)|x|^\alpha + 1)}{|x|^{2\alpha} + 2\cos(\alpha\pi/2)|x|^\alpha + 1}
\end{align*}
is odd and differentiable with respect to $|x|^\alpha$.
It suggests the change of variable $\tilde f(x) = f(x^{1/\alpha})$,
and the need to study the Hilbert transform
\begin{align*}
\widetilde{Hf}(x) = Hf(x^{1/\alpha})
&= \frac{1}{\pi}\int_{-\infty}^\infty \frac{f(y)dy}{x^{1/\alpha} - y}
= \frac{1}{\pi}\int_0^\infty \frac{2yf(y)dy}{x^{2/\alpha} - y^2}\\
&= \frac{1}{\alpha\pi}\int_0^\infty \frac{2z^{2/\alpha-1}\tilde f(z)dz}{x^{2/\alpha} - z^{2/\alpha}} \tag{with $f$ odd and $z = y^\alpha$}
\end{align*}
in the new variable. In terms of the new variable we need to solve
\[
F=(1 + \lambda)xW_x + a\widetilde{UW_x} + (1 - \widetilde{HW})\tilde W=0
, \quad U_x = HW.
\]
In Section \ref{Hilbert} we will generalize the estimates in \cite{ElJe} from the discrete range $1/\alpha \in \Z$
to the full range $\alpha \in (0, 1)$, using some delicate analysis of the integral kernel. Armed with these estimates, in Section \ref{Holder}
we will use the implicit function theorem to show that a perturbation of the explicit solution above exists as long as $|a\alpha|$ is small enough.
The argument mostly follows \cite{ElJe}, but is easier if we let
$F$ map $X := \{W \in H^2: xW_x \in H^2\}$ to $H^2$.
This way all terms in $F$ are automatically in $H^2$:
for the most difficult term $UW_x = (U/x)(xW_x)$ we will need the Hardy inequalities collected in Section \ref{Hardy}. Then we only need the $H^2 \to X$ bound of $(dF)^{-1}$ for the implicit function theorem to work and give us the desired solultion for small $|a\alpha|$.

\subsection{Acknowledgement}\label{Ack}

The research of the author was partially supported by ERC
(European Research Council) under Grant 788250.

\section{Hardy inequalities}\label{Hardy}

Here we record some useful Hardy-type inequalities in Sobolev spaces.

\begin{lem}[Kufner--Persson Theorem 4.3]\label{wtHardy}
If $p \in (1, \infty)$, $\gamma < (p - 1)/p$, $k \ge 1$ and $f(0) = \cdots = f^{(k-1)}(0) = 0$, then
\[
\|x^{\gamma-k}f(x)\|_{L^p} \lesssim_{k,p,\gamma} \|x^\gamma f^{(k)}\|_{L^p}.
\]
\end{lem}

\begin{df}
For any integer $k \ge 0$, any number $p \in (1, \infty)$ and any weight $w(x) \ge 0$ define
\[
\|f\|_{W^{k,p}(w)} = \sum_{j=0}^\infty \|wf^{(j)}\|_{L^p},\quad
\|f\|_{\bar W^{k,p}(w)} = \|f\|_{W^{k,p}(w)} + \|xf'(x)\|_{W^{k,p}(w)}.
\]
Let $W^{k,p}_-(w)$ and $\bar W^{k,p}_-(w)$ denote the subspace of odd functions in $W^{k,p}(w)$.
\end{df}

\begin{df}
\[
If(x) = \int_0^x \frac{f(y) - f(0) - yf'(0)}{y^2}dy.
\]
\end{df}

\begin{lem}\label{If'-Hk}
If $k \ge 0$ and $p \in (1, \infty)$ then
\[
\|(If)'\|_{W^{k,p}} \lesssim_{k,p,\gamma} \|f''\|_{W^{k,p}} \le \|f\|_{W^{k+2,p}}.
\]
\end{lem}
\begin{proof}
Without loss of generality assume $f(0) = f'(0) = 0$.
Then the result follows from Lemma \ref{wtHardy} because $(x^{-2}f(x))^{(k)}$ is a linear combination of $x^{-2-j}f^{(k-j)}(x)$, $0 \le j \le k$.
\end{proof}

\begin{lem}\label{gIf-Hk}
If $k \ge 0$ and $p \in (1, \infty)$ then
\begin{align*}
\left\| g(x)\int_0^x f(y)dy \right\|_{W^{k,p}}
&\lesssim_k \sum_{l=0}^k \left\| g^{(l)}(x)\int_0^x f(y)dy \right\|_{L^p}\\
&+ \sum_{\genfrac{}{}{0pt}{2}{n\ge1}{m+n\le k}} \|g^{(m)}f^{(n-1)}\|_{L^p}\\
&\lesssim_{k,p,\gamma} \sum_{l=0}^k \sup|xg^{(l)}(x)\|\|f\|_{L^p}\\
&+ 1_{k\ge1}\|g\|_{C^{k-1}}\|f\|_{W^{k-1,p}}
\tag{by Lemma \ref{wtHardy}}.
\end{align*}
\end{lem}
\begin{rem}
If we discount the term where all the derivatives fall on the integral,
we have $1 \le n \le k - 1$ in the summation, so we only need the $W^{k-2,p}$ norm of $f$.
\end{rem}

\begin{lem}\label{If-Hk}
If $k \ge 0$ and $p \in (1, \infty)$ then
\[
\left\| \frac{xIf(x)}{1 + x^2} \right\|_{W^{k,p}}
\lesssim_{k,p,\gamma} \|f'\|_{W^{\max(k-1,1),p}}
\le \|f\|_{W^{\max(k,2),p}}.
\]
\end{lem}
\begin{proof}
By Lemma \ref{gIf-Hk} (note that we have subtracted the term with all $k$ derivatives hitting $If$) and Lemma \ref{If'-Hk},
\begin{align*}
\left\| \frac{d^k}{dx^k}\frac{xIf(x)}{1 + x^2}
- 1_{k\ge1}\frac{x(If)^{(k)}(x)}{1 + x^2} \right\|_{L^p}
&\lesssim_{k,p,\gamma} \|(If)'\|_{W^{\max(k-2,0),p}}\\
&\lesssim_{k,p,\gamma} \|f'\|_{W^{\max(k-1,1),p}}.
\end{align*}
If $k = 0$ then there is nothing more to prove. If $k \ge 1$ then similarly \[
\left\| \frac{d^{k-1}}{dx^{k-1}}\frac{x(If)'(x)}{1 + x^2}
- 1_{k\ge1}\frac{x(If)^{(k)}(x)}{1 + x^2} \right\|_{L^p}
\lesssim_{k,p,\gamma} \|f'\|_{W^{\max(k-1,1),p}}.
\]
Also,
\begin{align*}
\left\| \frac{x(If)'(x)}{1 + x^2} \right\|_{W^{k-1,p}}
&= \left\| \frac{f(x) - f(0)}{(1 + x^2)x} - \frac{f'(0)}{1 + x^2} \right\|_{W^{k-1,p}}\\
&\lesssim_k \|1/(1 + x^2)\|_{C^{k-1}}\|(f(x) - f(0))/x\|_{W^{k-1,p}}\\
&+ \|1/(1 + x^2)\|_{W^{k-1,p}}|f'(0)|\\
&\lesssim_{k,p,\gamma} \|f'\|_{W^{\max(k-1,1),p}}.
\tag{by Lemma \ref{wtHardy} and $W^{1,p} \sub C^0$}\\
\end{align*}
\end{proof}

\section{Hilbert transform on H\"older functions}\label{Hilbert}

In this section we generalize the bounds for the Hilbert transform in \cite{ElJe}.

\begin{df}
For a function $f$ let $\tilde f(x) = f(x^{1/\alpha})$ ($x \ge 0$).
\end{df}

For example,
\[
\tilde{\bar W}(x) = -\frac{2\sin(\alpha\pi/2)x}{x^2 + 2\cos(\alpha\pi/2)x + 1}, \quad \widetilde{H\bar W}(x) = \frac{2(\cos(\alpha\pi/2)x + 1)}{x^2 + 2\cos(\alpha\pi/2)x + 1}.
\]

\begin{rem}
In this section we only consider functions defined on $\{x: x \ge 0\}$,
unless stated otherwise. 
\end{rem}

\begin{df}\label{Hr-df}
For $r \ge 1$ (not necessarily an integer) define
\[
H^{(r)}(f)(x)
= \frac{1}{\pi}\int_0^\infty \frac{2ry^{2r-1}}{x^{2r} - y^{2r}}f(y)dy
\]
and
\[
\tilde Hf = H^{(1/\alpha)}f
\]
so that
\[
\tilde Hf(x) = (HEf(\cdot^\alpha))(x^{1/\alpha})
\]
where $Ef(x) = f(|x|)\sgn x$ is the odd extension of $f$. Then for odd $V$,
\[
\tilde H\tilde V(x) = HEV(x^{1/\alpha}) = HV(x^{1/\alpha}) = \widetilde{HV}(x).
\]
\end{df}

Before bounding this operator, we need some elementary inequalities.

\begin{lem}\label{kernel-ineq}
If $r \ge 1$ and $t \ge 0$ then
\[
\frac{2rt^{2r-1}}{1 - t^{2r}} \le \frac{2t}{1 - t^2} \le \frac{2rt}{1 - t^{2r}} \le \frac{2rt^{2r-1}}{1 - t^{2r}} + 2r.
\]
\end{lem}
\begin{proof}
First we show the first two inequalities.
Clearing the denominator and canceling the factor $2t$ give
$rt^{2r-2} - rt^{2r} \le 1 - t^{2r} \le r - rt^2$.
The first inequality is equivalent to $(r - 1)t^{2r} + 1 \ge rt^{2r-2}$,
and the second equivalent to $t^{2r} + r - 1 \ge rt^2$.
Both follow from Young's inequality.
The last inequality is nothing but $|t - t^{2r-1}| \le |1 - t^{2r}|$,
which holds because $t$ and $t^{2r-1}$ are always between 1 and $t^{2r}$.
\end{proof}

\begin{lem}\label{Hn-L2}
For $r \ge 1$ we have $\|H^{(r)}\|_{L^2\to L^2} \le Cr$ for some constant $C$.
\end{lem}
\begin{proof}
We have
\[
H^{(1)}f(x) = \frac{1}{\pi}\int_0^\infty \frac{2yf(y)dy}{x^2 - y^2}
= \frac{1}{\pi}\int_0^\infty \left( \frac{1}{x - y} - \frac{1}{x + y} \right)f(y)dy = HEf(x)
\]
where $Ef(x) = f(|x|)\sgn x$, so $H^{(2)}$ is an isometry on $L^2(\R^+)$.

For general $r$ we have
\[
\pi H^{(r)}f(x)
= \int_0^\infty \frac{2ry^{2r-1}}{x^{2r} - y^{2r}}f(y)dy
= \int_0^\infty \frac{2rt^{2r-1}}{1 - t^{2r}}f(tx)dt \tag{$y = tx$}
\]
so
\[
\pi(H^{(r)} - H^{(1)})f(x) = \int_0^\infty \left( \frac{2rt^{2r-1}}{1 - t^{2r}} - \frac{2t}{1 - t^2} \right)f(tx)dt.
\]
Since $\|f(tx)\|_{L^2_x} = \|f\|_{L^2}/\sqrt t$, we have
\[
\|H^{(r)}\|_{L^2\to L^2}
\le 1 + \frac{1}{\pi}\int_0^\infty \left( \frac{2t}{1 - t^2} -  \frac{2rt^{2r-1}}{1 - t^{2r}} \right)\frac{dt}{\sqrt t}.
\]
Note that the integrand is nonnegative by Lemma \ref{kernel-ineq}.
By the same lemma,
\[
\int_0^2 \left( \frac{2t}{1 - t^2} -  \frac{2rt^{2r-1}}{1 - t^{2r}} \right)\frac{dt}{\sqrt t} \le 2r\int_0^2 \frac{dt}{\sqrt t} = 4\sqrt2r.
\]
For $t \ge 2$,
\[
\frac{2t}{1 - t^2} -  \frac{2rt^{2r-1}}{1 - t^{2r}}
\le \frac{2rt^{2r-1}}{t^{2r} - 1} \le \frac{8r}{3t}
\]
so
\[
\int_2^\infty \left( \frac{2t}{1 - t^2} -  \frac{2rt^{2r-1}}{1 - t^{2r}} \right)\frac{dt}{\sqrt t} \le \int_2^\infty \frac{8rdt}{3t\sqrt t}
= \frac{8\sqrt2}{3}r
\]
and then
\[
\|H^{(r)}\|_{L^2\to L^2} \le 1 + \frac{20\sqrt2}{3\pi}r
\le \left( 1 + \frac{20\sqrt2}{3\pi} \right)r.
\]
\end{proof}

\begin{lem}\label{Hn-f''}
For $r \ge 1$, if $f(0) = 0$ then for $k = 1, 2$ and $x \ne 0$,
\[
(H^{(r)}f)^{(k)}(x) = \frac{1}{\pi}\int_0^\infty \frac{2rx^{2r-k}y^{k-1}}{x^{2r} - y^{2r}}f^{(k)}(y)dy.
\]
\end{lem}
\begin{proof}
Using $H^{(1)}f = HEf$, where $Ef(x) = f(|x|)\sgn x$, we see that the identity holds for $r = 1$.

For $r > 1$ we have (note the singularity of the kernel has been subtracted)
\[
\pi((H^{(r)} - H^{(1)})f)'(x) = \int_0^\infty \partial_x\left( \frac{2ry^{2r-1}}{x^{2r} - y^{2r}} - \frac{2y}{x^2 - y^2} \right)f(y)dy.
\]
By Euler's theorem on homogeneous functions ($xF_x + yF_y = 0$ for $F$ homogeneous of degree 0) applied to the starred equality,
\begin{align*}
\partial_x\frac{y^{2r-1}}{x^{2r} - y^{2r}}
&= \frac{1}{y}\partial_x\frac{y^{2r}}{x^{2r} - y^{2r}}
\overset{*}{=} -\frac{1}{x}\partial_y\frac{y^{2r}}{x^{2r} - y^{2r}}
= -\frac{1}{x}\partial_y\frac{x^{2r}}{x^{2r} - y^{2r}}\\
&= -\partial_y\frac{x^{2r-1}}{x^{2r} - y^{2r}}
\end{align*}
so
\[
\pi((H^{(r)} - H^{(1)})f)'(x) = -\int_0^\infty \partial_y\left( \frac{2rx^{2r-1}}{x^{2r} - y^{2r}} - \frac{2x}{x^2 - y^2} \right)f(y)dy.
\]
Since $f(0) = 0$ we can integrate by parts to get the identity for $k = 1$.

Similarly,
\[
\partial_x\frac{x^{2r-1}}{x^{2r} - y^{2r}}
= \frac{1}{y}\partial_x\frac{x^{2r-1}y}{x^{2r} - y^{2r}}
= -\frac{1}{x}\partial_y\frac{x^{2r-1}y}{x^{2r} - y^{2r}}
= -\partial_y\frac{x^{2r-2}y}{x^{2r} - y^{2r}}
\]
so
\[
\pi((H^{(r)} - H^{(1)})f)''(x) = -\int_0^\infty \partial_y\left( \frac{2rx^{2r-2}y}{x^{2r} - y^{2r}} - \frac{2y}{x^2 - y^2} \right)f'(y)dy.
\]
Integrating by parts we get the identity for $k = 2$. This time we don't need $f'(0) = 0$ because the parenthesis vanishes at $y = 0$.
\end{proof}

\begin{lem}\label{Hn-f''-L2}
For $r \ge 1$, if $f(0) = 0$ then $\|(H^{(r)}f)''\|_{L^2} \le Cr\|f''\|_{L^2}$ for some constant $C$.
\end{lem}
\begin{proof}
Since $f(0) = 0$, taking the second derivative commutes with $E$ and $H$.
Then $(H^{(1)}f)'' = (HEf)'' = HEf'' = H^{(1)}(f'')$, so $\|(H^{(1)}f)''\|_{L^2} = \|H^{(1)}(f'')\|_{L^2} = \|f''\|_{L^2}$.

For $r > 1$, we change variable as before to get
\[
\|(H^{(r)}f)''\|_{L^2}
\le \left( 1 + \frac{1}{\pi}\int_0^\infty \left( \frac{2rt}{1 - t^{2r}} - \frac{2t}{1 - t^2} \right)\frac{dt}{\sqrt t} \right)\|f\|_{H^2}.
\]
Note that the integrand is nonnegative by Lemma \ref{kernel-ineq}.
By the same lemma,
\[
\int_0^2 \left( \frac{2rt}{1 - t^{2r}} - \frac{2t}{1 - t^2} \right)\frac{dt}{\sqrt t} \le 2r\int_0^2 \frac{dt}{\sqrt t} = 4\sqrt2r.
\]
For $t \ge 2$,
\[
\frac{2rt}{1 - t^{2r}} - \frac{2t}{1 - t^2} \le \frac{2t}{t^2 - 1} \le \frac{8}{3t}
\]
so
\[
\int_2^\infty \left( \frac{2rt}{1 - t^{2r}} - \frac{2t}{1 - t^2} \right)\frac{dt}{\sqrt t} \le \int_2^\infty \frac{8dt}{3t\sqrt t} = \frac{8\sqrt2}{3}
\]
and then
\[
\|(H^{(r)}f)''\|_{L^2} \le \left( 1 + \frac{1}{\pi}\left( 4\sqrt2r + \frac{8\sqrt2}{3} \right) \right)\|f''\|_{L^2} \le \left( 1 + \frac{20\sqrt2}{3\pi} \right)r\|f\|_{H^2}.
\]
\end{proof}

By Lemma \ref{Hn-L2} and Lemma \ref{Hn-f''-L2} we get
\begin{lem}\label{Hn-H2}
For $r \ge 1$ we have $\|H^{(r)}\|_{H^2_0\to H^2} \le Cr$ for some constant $C$, where $H^2_0$ denotes the space of $H^2$ functions vanishing at 0.
\end{lem}

\begin{rem}
To bound higher derivatives of $H^{(r)}f$ in $L^2$, more derivatives of $f$ at 0 need to vanish.
\end{rem}

\section{H\"older steady states for nonzero $a$}\label{Holder}

We first define the spaces which we are going to work with.

\begin{df}
Let $\bar H^2 = \{f \in H^2: xf_x \in H^2\}$ with $\|f\|_{\bar H^2} = \|f\|_{H^2} + \|xf_x\|_{H^2}$. We subscript a space by 0 to indicate the subspace of functions that vanish at 0. For example, $\bar H^2_0 = \{f \in \bar H^2: f(0) = 0\}$.
\end{df}

By Lemm 2.2 of \cite{ElJe}, for $0 < \alpha \le 1$ and $f \in H^2_0$ we have $(\tilde Hf)_x(0) = f_x(0)\cot(\alpha\pi/2)$.

\begin{df}
Let $X = \{V \in \bar H^2_0: V_x(0) = 0\}$ and $Y = \{V \in H_0^2: V_x(0) + 2\sin(\alpha\pi/2)\tilde HV(0) = 0\}$.
\end{df}

We will solve the steady-state equation $F = 0$ (see Section \ref{Method})
using the implicit function theorem, so we need to find its differential.
\[
dF_{(\bar W,0,0)}(V, \mu, 0)
= LV + \mu x\tilde{\bar W}_x
\]
where
\[
LV = V + xV_x - \tilde{\bar W}\tilde HV - V\tilde H\tilde{\bar W}.
\]

\begin{lem}\label{Lt-iso}
If $0 < \alpha \le 1$, then $L$ is an isomorphism from $X$ to $Y$ and
\begin{align*}
L^{-1}f(x)
&= \frac{x(1 - x^2)\sin(\alpha\pi/2)}{(1 + 2x\cos(\alpha\pi/2) + x^2)^2}\\
&\times \int_0^x \frac{1 - y^2}{y}\sin\frac{\alpha\pi}{2}g(y)dy
+ \left( \frac{1 + y^2}{y}\cos\frac{\alpha\pi}{2} + 2 \right)h(y)dy\\
&- \frac{x((1 + x^2)\cos(\alpha\pi/2) + 2x)}{(1 + 2x\cos(\alpha\pi/2) + x^2)^2}\\
&\times \int_0^x -\left( \frac{1 + y^2}{y}\cos\frac{\alpha\pi}{2} + 2 \right)g(y)
+ \frac{1 - y^2}{y}\sin\frac{\alpha\pi}{2}h(y)dy
\end{align*}
where
\[
g(x) = \frac{f(x)}{x} - \frac{f'(0)}{1 + 2x\cos(\alpha\pi/2) + x^2},\\
\]
and
\begin{align*}
h(x) &= \frac{\tilde Hf(x)}{x} - \frac{(1 - x^2)\tilde Hf(0)}{x(1 + 2x\cos(\alpha\pi/2) + x^2)}\\
&= \frac{\tilde Hf(x) - \tilde Hf(0)}{x} + \frac{2(\cos(\alpha\pi/2) + x)}{1 + 2x\cos(\alpha\pi/2) + x^2}\tilde Hf(0)\\
&= \frac{\tilde Hf(x) - \tilde Hf(0) - x(\tilde Hf)'(0)}{x} - \frac{2x(\cos\alpha\pi + x\cos(\alpha\pi/2))}{1 + 2x\cos(\alpha\pi/2) + x^2}\tilde Hf(0).
\end{align*}
The norm of $L^{-1}$ is bounded, uniformly in $\alpha$.
\end{lem}

This will be proved towards the end of the section.

\begin{lem}
For $0 < \alpha \le 1$ and $n = 0, \pm1$,
\begin{align*}
T_{n,\alpha}f &= \frac{x}{1 + x^2}\int_0^x y^n\left( \frac{f(y)}{y} - \frac{f'(0)}{1 + 2y\cos(\alpha\pi/2) + y^2} \right)dy,\\
S_{n,\alpha}f &= \frac{x^{1-\max(n,0)}}{1 + x^2}\int_0^x y^n\left( \frac{\tilde Hf(y) - \tilde Hf(0)}{y} - \frac{2(\cos(\alpha\pi/2) + y)\tilde Hf(0)}{1 + 2y\cos(\alpha\pi/2) + y^2} \right)dy
\end{align*}
are bounded from $Y$ to $H^2$, with $\|T_{n,\alpha}\|_{Y\to H^2} \le C$ and
$\|S_{n,\alpha}\|_{Y\to H^2} \le C/\alpha$.
\end{lem}
\begin{proof}
Clearly $T_{n,\alpha}f(0) = S_{n,\alpha}f(0) = 0$, so it remains to bound the $H^2$ norm.

For $T_{n,\alpha}$ we have
\begin{align*}
(T_{n,1} - T_{n,\alpha})f(x)
&= \frac{xf'(0)}{1 + x^2}\int_0^x \frac{2(1 - \cos(\alpha\pi/2))y^{n+1}}
{(1 + y^2)(1 + 2y\cos(\alpha\pi/2) + y^2)}dy\\
&= \frac{xf'(0)}{1 + x^2}g_{n,\alpha}(x)
\end{align*}
where all higher derivatives of $g_{n,\alpha}$ are bounded,
uniformly in $\alpha \in [0, 1]$. $T_{n,1}$ has the desired bound,
as shown in (4.15) of \cite{ElJe}, so does $T_{n,\alpha}$ because $x/(1 + x^2) \in H^\infty$ and $|f'(0)| \lesssim \|f\|_{H^2}$.

For $S_{-1,\alpha}$, we rewrite the integrand as
\[
I\tilde Hf(y) - \frac{2(\cos\alpha\pi + y\cos(\alpha\pi/2))}{1 + 2y\cos(\alpha\pi/2) + y^2}\tilde Hf(0).
\]
For the first term, by Lemma \ref{If-Hk} and Lemma \ref{Hn-H2},
\[
\left\| \frac{xI\tilde Hf(x)}{1 + x^2} \right\|_{H^2} \lesssim \|\tilde Hf\|_{H^2} = \|H^{(1/\alpha)}f\|_{H^2} \lesssim \|f\|_{H^2}/\alpha.
\]
For the second term we have
\begin{align*}
\frac{x}{1 + x^2}&\int_0^x \frac{2(\cos\alpha\pi + y\cos(\alpha\pi/2))}{1 + 2y\cos(\alpha\pi/2) + y^2}dy\\
= \frac{x}{1 + x^2}&\left( \cos\frac{\alpha\pi}{2}\ln(1 + 2x\cos\frac{\alpha\pi}{2} + x^2) - 2\sin\frac{\alpha\pi}{2}\arctan\frac{x\sin\frac{\alpha\pi}{2}}{1 + x\cos\frac{\alpha\pi}{2}} \right)
\end{align*}
whose $\bar H^2$ norm is bounded, uniformly in $\alpha$, so the second term has the desired bound because $|\tilde Hf(0)| \lesssim \|\tilde Hf\|_{H^1} \lesssim \|f\|_{H^1}/\alpha$.

For $S_{n,\alpha}$ ($n = 0, 1$) we use the old integrand. By Lemma \ref{gIf-Hk}, Hardy's inequality and Lemma \ref{Hn-H2}, the contribution of the first term in $S_{0,\alpha}$ is
\[
\lesssim \left\| \frac{\tilde Hf(x) - \tilde Hf(0)}{x} \right\|_{H^1} + \|(\tilde Hf)'\|_{H^1} \lesssim \|\tilde Hf\|_{H^2}
\lesssim \|f\|_{H^2}/\alpha
\]
and the contribution of the first term in $S_{1,\alpha}$ is
\[
\lesssim \|x/(1 + x^2)\|_{C^2}\left\| \frac{1}{x}\int_0^x \tilde Hf(y)dy \right\|_{H^2} \lesssim \|\tilde Hf\|_{H^2} \lesssim \|f\|_{H^2}/\alpha.
\]
For the second term in $S_{0,\alpha}$ we have
\[
\frac{x}{1 + x^2}\int_0^x \frac{2(\cos(\alpha\pi/2) + y)}{1 + 2y\cos(\alpha\pi/2) + y^2}dy
= \frac{x\ln(1 + 2x\cos(\alpha\pi/2) + x^2)}{1 + x^2}
\]
and
\begin{align*}
\frac{1}{1 + x^2}&\int_0^x \frac{2y(\cos(\alpha\pi/2) + y)}{1 + 2y\cos(\alpha\pi/2) + y^2}dy\\
= \frac{1}{1 + x^2}&\left(2x - \int_0^x \frac{2y\cos(\alpha\pi/2)}{1 + 2y\cos(\alpha\pi/2) + y^2}dy \right)\\
= \frac{1}{1 + x^2}&\left(2x - \cos\frac{\alpha\pi}{2}\ln\left( 1 + 2x\cos\frac{\alpha\pi}{2} + x^2 \right) \right)\\
+ \frac{2}{1 + x^2}&\frac{\cos^2\frac{\alpha\pi}{2}}{\sin\frac{\alpha\pi}{2}}\arctan\frac{x\sin\frac{\alpha\pi}{2}}{1 + x\cos\frac{\alpha\pi}{2}}
\end{align*}
whose $H^2$ norms arae bounded, uniformly in $\alpha$, so they also have the desired bound as before.
\end{proof}

Then we upgrade the $H^2$ norm to the $\bar H^2$ norm.
\begin{lem}
For $0 < \alpha \le 1$, $n = 0, \pm1$, $T_{n,\alpha}$ and $S_{n,\alpha}$ are bounded from $Y$ to $\bar H^2_0$, with $\|T_{n,\alpha}\|_{Y\to\bar H^2} \le C$ and $\|S_{n,\alpha}\|_{Y\to\bar H^2} \le C/\alpha$.
\end{lem}
\begin{proof}
Since
\[
(T_{n,1} - T_{n,\alpha})f = \frac{xf'(0)}{1 + x^2}g_{n,\alpha}(x)
\]
where all higher derivatives of $xg_{n,\alpha}'$ are bounded,
$T_{n,\alpha}$ can be controlled as before.

For $S_{n,\alpha}$, as before we can fall the derivative on the integral and it remains to bound
\[
\frac{x^{\min(n,0)+1}}{1 + x^2}\left( \tilde Hf(x) - \tilde Hf(0) - \frac{2x(\cos(\alpha\pi/2) + x)\tilde Hf(0)}{1 + 2x\cos(\alpha\pi/2) + x^2} \right)
\]
whose $H^2$ norm is $\lesssim \|\tilde Hf\|_{H^2} \lesssim \|f\|_{H^2}/\alpha$ by Lemma \ref{Hn-H2}.
\end{proof}

\begin{proof}[Proof of Lemma \ref{Lt-iso}]
The formal expression has been derived in Section 4.2 of \cite{ElJe}.
The $\bar H^2$ bound of $L^{-1}f$ comes from those of $T_{n,\alpha}f$ and $S_{n,\alpha}f$. Note that the only terms in $L^{-1}$ not covered by them are
\[
\frac{-x^3\sin(\alpha\pi/2)}{(1 + 2x\cos(\alpha\pi/2) + x^2)^2}
\int_0^\infty y\cos(\alpha\pi/2)h(y)dy
\]
and
\[
-\frac{x^3\cos(\alpha\pi/2)}{(1 + 2x\cos(\alpha\pi/2) + x^2)^2}
\int_0^\infty -y\sin(\alpha\pi/2)h(y)dy
\]
which cancel each other. To show the bound is uniform in $\alpha$,
it suffices to note that each appearance of $S_{n,\alpha}$ (via $h$) in $L^{-1}$ is acompanied by a factor of $\sin(\alpha\pi/2) < 2\alpha$,
which offsets the worse bounds of $S_{n,\alpha}$.
Finally, $L^{-1}f(0) = (L^{-1}f)'(0) = 0$ because the integrals vanish at 0, and they are multiplied by a factor of $x$.
\end{proof}

Recall that in terms of the new variable we need to solve
\[
F:=(1 + \lambda)xW_x + a\widetilde{UW_x} + (1 - \tilde H\tilde W)\tilde W=0
, \quad U_x = HW.
\]
\begin{lem}\label{analytic}
$F(W, \lambda, a): \bar H^2_0 \times \R \times \R \to H^2_0$ is analytic.
\end{lem}
\begin{proof}
Multiplicative closedness of $H^2$ takes care of every term but $\widetilde{UW_x}$, which we deal with now.
Since $UW_x = (U/x)(xW_x)$, in terms of the variable $y = x^\alpha$ we have
$\widetilde{UW_x}(y) = (\tilde U(y)/y^{1/\alpha})(\alpha y\tilde W_y)$.
For the second factor we have $\|y\tilde W_y\|_{H^2} \le \|\tilde W\|_{\bar H^2}$.
To bound the first factor we start from the identity $xU_x = xHW$.
In terms of $y$ it is $\alpha y\tilde U_y = y^{1/\alpha}\widetilde{HW}$, so
\[
\frac{\tilde U(y)}{y^{1/\alpha}} = \frac{1}{y^{1/\alpha}}\int_0^y \frac{\widetilde{HW}(z)}{\alpha z^{1-1/\alpha}}dz.
\]
Taking the derivative and integrating by parts we get
\[
\frac{d}{dy}\frac{\tilde U(y)}{y^{1/\alpha}}
= \frac{\widetilde{HW}(y)}{\alpha y} - \frac{1}{\alpha y^{1+1/\alpha}}
\int_0^y \frac{\widetilde{HW}(z)}{\alpha z^{1-1/\alpha}}dz
= \frac{1}{\alpha y^{1+1/\alpha}}\int_0^y \widetilde{HW}'(z)z^{1/\alpha}dz
\]
and similarly,
\[
\frac{d^2}{dy^2}\frac{\tilde U(y)}{y^{1/\alpha}}
= \frac{1}{\alpha y^{2+1/\alpha}}\int_0^y \widetilde{HW}''(z)z^{1+1/\alpha}dz.
\]
By Lemma \ref{wtHardy},
\[
\|\tilde U(y)/y^{1/\alpha}\|_{L^2}\lesssim C_\alpha\|\widetilde{HW}\|_{L^2}/\alpha, \quad
\|(\tilde U(y)/y^{1/\alpha})''\|_{L^2}\lesssim C_\alpha'\|\widetilde{HW}''\|_{L^2}/\alpha.
\]
The order of $C_\alpha$ and $C_\alpha''$ can be found in Chapter 0 of Kufner--Persson: Since the weight are $y^{\pm2-2/\alpha}$,
it corresponds to $\ep = \pm2 - 2/\alpha$ (and $p = 2$) in (0.4),
so $C_\alpha\lesssim 1/(1 - \ep) \le 1/(2/\alpha - 1) \le \alpha$ for $\alpha \in (0, 1)$. Hence
\[
\|\tilde U(y)/y^{1/\alpha}\|_{H^2}\lesssim \|\widetilde{HW}\|_{H^2}
\lesssim \|\tilde W\|_{H^2}/\alpha
\]
by Lemma \ref{Hn-H2}, so
\[
\|\widetilde{UW_x}\|_{H^2} \lesssim \|\tilde W\|_{\bar H^2} \cdot \alpha\|\tilde W\|_{H^2}/\alpha \le \|\tilde W\|_{\bar H^2}^2.
\]
\end{proof}

In particular,
\[
dF_{(\tilde{\bar W},0,0)}(V, \mu, 0) = \mu x\tilde{\bar W}_x + xV_x - \tilde{\bar W}\tilde HV - V\tilde H\tilde{\bar W} + V = LV + \mu x\tilde{\bar W}_x.
\]

\begin{lem}\label{dF-iso}
$dF_{(\tilde{\bar W},0,0)}(\cdot,\cdot,0)$ is an isomorphism from $X \times \R$ to $H^2_0$.
\end{lem}
\begin{proof}
It suffices to show that $(x\tilde{\bar W}_x)_x + 2\tilde H(x\tilde{\bar W}_x)$ does not vanish at 0. This is because
$(x\tilde{\bar W}_x)_x(0) = \tilde{\bar W}_x(0) = -2\sin(\alpha\pi/2)$ and
$\tilde H(x\tilde{\bar W}_x)(0) = H(x\bar W_x)/\alpha = 0$,
\end{proof}

\begin{proof}[Proof of Theorem \ref{Thm}]
Let $W = \alpha\Omega$, $U = \alpha\Upsilon$ and $F = \alpha\Phi$. Then
\[
\Phi = (1 + \lambda)x\tilde\Omega_x + a\alpha\widetilde{\Upsilon\Omega_x} + (1 - \alpha\tilde H\tilde\Omega)\Omega.
\]
Note that every term in $\Phi$ is analytic from $\bar H^2_0$ to $H^2_0$,
with all derivatives uniform in $\alpha$ (the bound for $\alpha\tilde H\tilde\Omega$ comes from Lemma \ref{Hn-H2}. Also
\[
d\Phi_{(\tilde{\bar\Omega},0,0)}(V, \mu, 0) = LV + \mu x\tilde{\bar W}_x/\alpha
\]
where
\[
\frac{(x\tilde{\bar W}_x)_x + 2\tilde H(x\tilde{\bar W}_x)}{\alpha}
= -2\frac{\sin(\alpha\pi/2)}{\alpha}
\]
is bounded and bounded away from 0, uniformly in $\alpha$.
Hence $d\Phi_{(\tilde{\bar\Omega},0,0)}(\cdot, \cdot, 0)$ is invertible
and its inverse is analytic, with all derivatives bounded,
uniformly in $\alpha$. Now by the implicit function theorem,
there is $c > 0$ such that if $|a\alpha| < c$, then there are $W(\cdot;a) \in \bar H^2_0$ and $\lambda(a) \in \R$ such that $\Phi(W(\cdot;a), \lambda(a), a) = 0$. Then
\[
w(x, t) = \frac{\sgn x}{1 - t}W\left( \left| \frac{x}{(1 - t)^{1+\lambda(a)}} \right|^\alpha;a \right)
\]
is a self-similar solution. Since $W(\cdot;\alpha) \in \bar H^2_0\sub C^1$,
$w(\cdot, t) \in C^\alpha$. Finally, the analyticity of $W(\cdot;a)$ and $\lambda(a)$ in $a$ follows from Lemma \ref{analytic}.
\end{proof}

\end{document}